\documentclass[twoside, 11pt]{article}
\usepackage{fancyhdr}
\usepackage{titlesec}
\usepackage{cite}
\usepackage{ifthen}
\usepackage{amsthm,amsmath,amsfonts}
\usepackage{indentfirst}

\titleformat{\section}{\large\bfseries}{\arabic{section}}{1em}{}
\newboolean{first}
\setboolean{first}{true}

\renewcommand\baselinestretch{1.2}

\newtheorem{thm}{Theorem}[section]
\newtheorem{rem}{Remark}[section]
\newtheorem{lem}{Lemma}[section]
\newtheorem{cor}{Corollary}[section]

\newtheorem{defn}{Definition}[section]

\theoremstyle{definition}
 \numberwithin{equation}{section}

\begin{document}

\setlength\abovedisplayskip{2pt}
\setlength\abovedisplayshortskip{0pt}
\setlength\belowdisplayskip{2pt}
\setlength\belowdisplayshortskip{0pt}

\vspace{8true mm}

\renewcommand{\baselinestretch}{1.6}\baselineskip 15pt

\begin{center}
\noindent{\LARGE\bf
Sharp Estimates for $p$-Adic Hardy, Hardy-Littlewood-P\'olya
Operators and Commutators}

{\normalsize  Zunwei Fu$^1$, Qingyan Wu$^1$ and Shanzhen Lu$^2$
\noindent\footnotetext{\baselineskip 10pt This work was partially supported by
 NSF of China (Grant Nos. 10871024, 10901076 and 10931001), NSF of Shandong Province
 (Grant Nos. Q2008A01 and ZR2010AL006) and the Key Laboratory of Mathematics and Complex System (Beijing Normal
University), Ministry of Education, China.)}}
\end{center}
 \vspace{0.2 true cm}

\begin{center}
\renewcommand{\baselinestretch}{1.3}\baselineskip 12pt
\noindent{\footnotesize\rm  $^1$ Department of Mathematics, Linyi Universtiy, Linyi 276005, P. R. China\\
E-mail: lyfzw@tom.com \quad wuqingyan@lyu.edu.cn\\
 $^2$ School of Mathematical Sciences, Beijing Normal University, Beijing 100875, P. R. China\\
E-mail: lusz@bnu.edu.cn
\vspace{4mm}}%
\end{center}

\begin{center}
\begin{minipage}{135mm}
{\bf \small Abstract}\hskip 2mm {\small In this paper we get the sharp estimates of the $p$-adic Hardy and
 Hardy-Littlewood-P\'olya  operators on $L^q(|x|^{\alpha}_pdx)$. Also, we prove that the commutators
generated by the $p$-adic Hardy operators (Hardy-Littlewood-P\'olya operators)
 and the central BMO functions are bounded on $L^q(|x|^{\alpha}_pdx)$, more generally, on Herz spaces.
 }
{\bf \small MSC(2010)}\hskip 2mm {\small  Primary: 47G10, 11F85; Secondary: 26D15, 47A30, 47A63.}
\end{minipage}
\end{center}

\section{Introduction}

For a prime number $p$, let $\mathbb{Q}_{p}$ be the field of $p$-adic numbers. It is
defined as the completion of the field of rational numbers $\mathbb{Q}$ with respect to the
non-Archimedean $p$-adic norm $|\cdot|_p$. This norm is defined as follows: $|0|_p=0$; If any non-zero
rational number $x$ is represented as $x=p^{\gamma}\frac{m}{n}$, where $m$ and $n$ are integers which
are not divisible by $p$ and $\gamma$ is an integer, then $|x|_p=p^{-\gamma}$. It is not difficult to
show that the norm satisfies the following properties:
$$|xy|_p =|x|_p|y|_p,\qquad |x + y|_p \leq \max\{|x|_p, |y|_p\}.$$
It follows from the second property that when $|x|_p \neq |y|_p$, then
$|x + y|_p = \max\{|x|_p, |y|_p\}$.
From the standard $p$-adic analysis \cite{Vl}, we see that any non-zero $p$-adic number $x\in\mathbb{Q}_p$
can be uniquely represented in the canonical series
\begin{equation}
x=p^{\gamma}\sum_{j=0}^{\infty}a_jp^j,\qquad \gamma=\gamma(x)\in\mathbb{Z},\label{xp}
\end{equation}
where $a_j$ are integers, $0\leq a_j\leq p-1$, $a_0\neq 0$. The series \eqref{xp} converges in the $p$-adic norm
because $|a_jp^j|_p=p^{-j}$. Set $\mathbb{Q}_{p}^{*}=\mathbb{Q}_{p}\setminus\{0\}$.

The space $\mathbb{Q}_{p}^{n}$ consists of points $x=(x_1,x_2,\cdots,x_n)$, where
$x_j\in\mathbb{Q}_{p},~j=1,2,\cdots,n$.
The $p$-adic norm on $\mathbb{Q}_{p}^{n}$ is
\begin{equation}
|x|_p:=\max_{1\leq j\leq n}|x_j|_p,\qquad x\in \mathbb{Q}_{p}^{n}.
\end{equation}
Denote by
$$B_\gamma(a)=\{x\in\mathbb{Q}_{p}^{n}:|x-a|_p\leq p^{\gamma}\},$$
 the ball with center at $a\in\mathbb{Q}_{p}^{n}$
and radius $p^{\gamma}$, and its boundary
$$S_\gamma(a)=\{x\in\mathbb{Q}_{p}^{n}:|x-a|_p=p^{\gamma}\}=B_\gamma(a)\setminus B_{\gamma-1}(a).$$
Since $\mathbb{Q}_{p}^{n}$ is a locally compact commutative group under addition,
it follows from the standard analysis that there exists a Haar measure $dx$ on $\mathbb{Q}_{p}^{n}$,
which is unique up to positive constant multiple and is translation invariant. We normalize the measure
$dx$ by the equality
$$\int_{B_0(0)}dx=|B_0(0)|_H=1,$$
 where $|E|_H$ denotes the Haar measure of
a measurable subset $E$ of $\mathbb{Q}_{p}^{n}$. By simple calculation, we can obtain that
$$|B_\gamma(a)|_H=p^{\gamma n},\qquad |S_\gamma(a)|_H=p^{\gamma n}(1-p^{-n}),$$
 for any $a\in\mathbb{Q}_{p}^{n}$.
For a more complete introduction to the $p$-adic field, see \cite{Vl} or \cite{Ta}.

In recent years, $p$-adic analysis has received a lot of attention due to its application
in Mathematical Physics (cf.\cite{Al}, \cite{Av},\cite{Kh1},\cite{Kh2},\cite{Va} and \cite{Vl}).
There are numerous papers on $p$-adic analysis, such as \cite{Ha1} and \cite{Ha2} about Riesz potentials,
\cite{Al1},\cite{Cc},\cite{Ch},\cite{Ko} and \cite{Zu} about $p$-adic pseudo-differential equations, etc.
The Harmonic Analysis on $p$-adic field has been drawing more and more concern
 ( cf. \cite{Ki1},\cite{Ki2},\cite{Rim},\cite{Ro1},\cite{Ro2} and references therein).

The well-known Hardy's integral inequality \cite{Ha} tells us that for $1<q<\infty$,
$$\|Hf\|_{L^q(\mathbb{R}^{+})}\leq \frac{q}{q-1}\|f\|_{L^q(\mathbb{R}^{+})},$$
where
the classical Hardy operator is defined by
$$Hf(x):=\frac{1}{x}\int_{0}^{x}f(t)dt,$$
for non-negative integral function $f$ on $\mathbb{R}^{+}$, and
the constant $\frac{q}{q-1}$ is the best possible. Thus the norm of Hardy operator on
$L^q(\mathbb{R}^{+})$ is
$$\|H\|_{L^q(\mathbb{R}^{+})\rightarrow L^q(\mathbb{R}^{+})}=\frac{q}{q-1}.$$

Faris\cite{Fa} introduced the following $n$-dimensional Hardy operator, for nonnegative function $f$ on
$\mathbb{R}^{n}$,
$$\mathcal{H}f(x):=\frac{1}{\Omega_n|x|^{n}}\int_{|y|<|x|}f(y)dy,\quad
\mathcal{H^*}f(x):=\frac{1}{\Omega_n}\int_{|y|\geq|x|}\frac{f(y)}{|y|^{n}}dy,\quad
x\in\mathbb{R}^n\setminus\{0\},$$
where $\Omega_n$ is the volume of the unit ball in $\mathbb{R}^n$. Christ and Grafakos\cite{CG} obtained that
the norm of $\mathcal{H}$ and $\mathcal{H^*}$ on $ L^q(\mathbb{R}^{n}) $ are
$$\|\mathcal{H}\|_{L^q(\mathbb{R}^{n})\rightarrow L^q(\mathbb{R}^{n})}=
\|\mathcal{H^*}\|_{L^q(\mathbb{R}^{n})\rightarrow L^q(\mathbb{R}^{n})}=\frac{q}{q-1},$$
which is the same as that of the 1-dimension Hardy operator. Obviously, $\mathcal{H}$ and $\mathcal{H^*}$ satisfy
$$\int_{\mathbb{R}^n}g(x)\mathcal{H}f(x)dx=\int_{\mathbb{R}^n}f(x)\mathcal{H^*}g(x)dx.
$$
 In \cite{Fu1}, Fu, Grafakos,
Lu and Zhao proved that $\mathcal{H}$ is also bounded on the weighted Lebesgue space $L^q(|x|^{\alpha}dx) $
for $1<q<\infty $, $\alpha<n(q-1) $. And
$$\|\mathcal{H}\|_{L^q(|x|^{\alpha}dx)\rightarrow L^q(|x|^{\alpha}dx)}=\frac{qn}{qn-n-\alpha}.$$
It is clear that when $\alpha=0 $, the
result coincides with that we mentioned above. Inspired by these results, in this paper we will introduce the definition of the $p$-adic Hardy
operator and establish the sharp estimates of their boundedness on $L^q(|x|^\alpha_p dx)$.

\begin{defn}\label{th:def}
For a function $f$ on $\mathbb{Q}_{p}^{n}$, we define the {\it $p$-adic Hardy operators} as follows
\begin{equation}\begin{split}
\mathcal{H}^pf(x)&=\frac{1}{|x|_p^{n}}\int_{B(0,|x|_p)}f(t)dt,\quad x\in\mathbb{Q}_{p}^{n}\setminus\{0\},\\
\mathcal{H}^{p,*}f(x)&=\int_{\mathbb{Q}_{p}^{n}\setminus B(0,|x|_p)}\frac{f(t)}{|t|_p^{n}}dt,
\quad x\in\mathbb{Q}_{p}^{n}\setminus\{0\},
\end{split}\end{equation}
where $B(0,|x|_p)$ is a ball in $\mathbb{Q}_{p}^{n}$ with center at $0\in\mathbb{Q}_{p}^{n} $
and radius $|x|_p$.
\end{defn}

The $p$-adic Hardy operators $\mathcal{H}^p$ and $\mathcal{H}^p$ are adjoint mutually:
$$\int_{\mathbb{Q}_{p}^{n}}g(x)\mathcal{H}^pf(x)dx
=\int_{\mathbb{Q}_{p}^{n}}f(x)\mathcal{H}^{p,*}g(x)dx,
$$
when $f\in L^q(\mathbb{Q}_{p}^{n})$ and $g\in L^{q'}(\mathbb{Q}_{p}^{n})$, $\frac{1}{q}+\frac{1}{q'}=1$.

It is obvious that $|\mathcal{H}^p f|\leq \mathcal{M}^p f$, where $\mathcal{M}^p $
is the Hardy-Littlewood maximal operator\cite{Ki1}
defined by
$$\mathcal{M}^p f(x)=\sup_{\gamma\in\mathbb{Z}}\frac{1}{|B_\gamma (x)|_H}
\int_{B_\gamma (x)}|f(y)|dy,
\quad f\in L_{loc}^{1}(\mathbb{Q}_{p}^{n}).$$
The Hardy-Littlewood maximal operator plays a very
important role in Harmonic Analysis. The boundedness of $\mathcal{M}^p$ on $L^q(\mathbb{Q}_{p}^n)$ has been solved
(see for instance \cite{Ta}). But the best estimate of $\mathcal{M}^p$ on $L^q(\mathbb{Q}_{p}^n),~q>1,$
even that of Hardy-Littlewood maximal operator on Euclidean spaces $\mathbb{R}^n$ is very difficult to obtain. Instead,
in Section 2, we obtain the sharp estimates of $\mathcal{H}^p$ (and $p$-adic Hardy-Littlewood-P\'olya
operator), and the norm of $\mathcal{M}^p$ should be no less than that of $\mathcal{H}^p$. In Section 3, we will define the commutators
of the $p$-adic Hardy and
 Hardy-Littlewood-P\'olya  operators, and discuss the boundedness of them. One of the main innovative points is that we estimate the commutator
of Hardy-Littlewood-P\'olya  operator by that
of the $p$-adic Hardy operator.

\section{Sharp estimates of $p$-adic Hardy and Hardy-Littlewood-P\'olya operators }

We get the following sharp estimates of  $\mathcal{H}^p$ and  $\mathcal{H}^{p,*}$ on $L^q(|x|_p^{\alpha}dx)$.

\begin{thm}\label{th:thm1}
Let $1<q<\infty$ and $\alpha<n(q-1)$. Then
\begin{equation}
\|\mathcal{H}^p\|_{L^q(|x|_p^{\alpha}dx)\rightarrow L^q(|x|_p^{\alpha}dx)}
=\|\mathcal{H}^{p,*}\|_{L^q(|x|_p^{\alpha}dx)\rightarrow L^q(|x|_p^{\alpha}dx)}
=\frac{1-p^{-n}}{1-p^{\frac{\alpha}{q}-\frac{n}{q'}}},\label{thm1}
\end{equation}
where $\frac{1}{q}+\frac{1}{q'}=1$.
\end{thm}

When $\alpha=0$, we get the following corollary.

\begin{cor}
Let $1<q<\infty$. Then
\begin{equation}
\|\mathcal{H}^p\|_{L^q(\mathbb{Q}_{p}^{n})\rightarrow L^q(\mathbb{Q}_{p}^{n})}
=\|\mathcal{H}^{p,*}\|_{L^q(\mathbb{Q}_{p}^{n})\rightarrow L^q(\mathbb{Q}_{p}^{n})}
=\frac{1-p^{-n}}{1-p^{-\frac{n}{q'}}},
\end{equation}
where $\frac{1}{q}+\frac{1}{q'}=1$.
\end{cor}

\vspace{0.2cm}
\noindent{\bf Remark.} {\it Obviously, the $L^q$ norm of $\mathcal{H}^p$ on $\mathbb{Q}_{p}^{n}$
depends on $n$, however, the $L^q$ norm of $\mathcal{H}$ on $\mathbb{R}^{n}$ is independent of
the dimension $n$.}

\vspace{0.4cm}
The Hardy-Littlewood-P\'olya's linear operator\cite{Be} is defined by
\begin{equation*}
Tf(x)=\int_{0}^{+\infty}\frac{f(y)}{\max(x,y)}dy.
\end{equation*}
In \cite{Be}, the authors obtained that the norm of Hardy-Littlewood-P\'olya's  operator on
$L^q(\mathbb{R}^{+})$ (see also P. 254 in \cite{Har} ), $1<q<\infty$, is
$$\|T\|_{L^q(\mathbb{R}^{+})\rightarrow L^q(\mathbb{R}^{+})}=\frac{q^2}{q-1}.$$

Next, we consider $p$-adic version of Hardy-Littlewood-P\'olya operator.
We define the $p$-adic  Hardy-Littlewood-P\'olya operator as
\begin{equation*}
T^pf(x)=\int_{\mathbb{Q}_{p}^{*}}\frac{f(y)}{\max(|x|_p,|y|_p)}dy, \qquad x\in\mathbb{Q}_{p}.
\end{equation*}
By the similar method to the proof of Theorem \ref{th:thm1}, we obtain the norm of $p$-adic  Hardy-Littlewood-
P\'olya operator from
$L^q(|x|^{\alpha}_pdx)$ to $L^q(|x|^{\alpha}_pdx)$.

\begin{thm}\label{th:thm3}
Let $1<q<\infty$ and $-1<\alpha<q-1$. Then for any $f\in L^q(|x|^{\alpha}_pdx)$,
\begin{equation}
\|T^pf\|_{L^q(|x|^{\alpha}_pdx)}\leq \left(1-\frac{1}{p}\right)
\left(\frac{1}{1-p^{\frac{\alpha+1}{q}-1}}
+\frac{p^{-\frac{\alpha+1}{q}}}{1-p^{-\frac{\alpha+1}{q}}}\right)
\|f\|_{L^q(|x|^{\alpha}_pdx)}.\label{eq}
\end{equation}
Moreover,
\begin{equation}
\|T^p\|_{L^q(|x|^{\alpha}_pdx)\rightarrow L^q(|x|^{\alpha}_pdx)}
=\left(1-\frac{1}{p}\right)
\left(\frac{1}{1-p^{\frac{\alpha+1}{q}-1}}
+\frac{p^{-\frac{\alpha+1}{q}}}{1-p^{-\frac{\alpha+1}{q}}}\right).\label{norm1}
\end{equation}
\end{thm}

When $\alpha=0$, we get the norm of $T^p$ on
$L^q(\mathbb{Q}_p)$.
\begin{cor}
Let $1<q<\infty$. Then
\begin{equation}
\|T^p\|_{L^q(\mathbb{Q}_p)\rightarrow L^q(\mathbb{Q}_p)}
=\left(1-\frac{1}{p}\right)
\left(\frac{1}{1-p^{\frac{1}{q}-1}}
+\frac{p^{-\frac{1}{q}}}{1-p^{-\frac{1}{q}}}\right).
\end{equation}
\end{cor}

\begin{proof}[Proof of Theorem \ref{th:thm1}]
Firstly, we claim that the operator $\mathcal{H}^p$ and its restriction to the functions $g$ satisfying $g(x)=g(|x|_p^{-1})$
 have the same operator norm on
$L^q(|x|_p^{\alpha}dx)$
 In fact, set
$$g(x)=\frac{1}{1-p^{-n}}\int_{|\xi|_p=1}f(|x|_p^{-1}\xi)d\xi,\qquad x\in \mathbb{Q}_p^{n}.$$
It's easy to see that $g$ satisfies that $g(x)=g(|x|_p^{-1})$ and $\mathcal{H}^pg=\mathcal{H}^pf$. By H\"older's inequality, we get
\begin{equation*}\begin{split}
\|g\|_{L^q(|x|_p^{\alpha}dx)}&=\left(
\int_{\mathbb{Q}_p^{n}}\left|\frac{1}{1-p^{-n}}\int_{|\xi|_p=1}f(|x|_p^{-1}\xi)d\xi\right|^q
|x|_p^{\alpha}dx\right)^{\frac{1}{q}}\\
&\leq\left(
\int_{\mathbb{Q}_p^{n}}\frac{1}{(1-p^{-n})^q}\left(\int_{|\xi|_p=1}|f(|x|_p^{-1}\xi)|^q d\xi\right)
\left(\int_{|\xi|_p=1} d\xi\right)^\frac{q}{q'}
|x|_p^{\alpha}dx\right)^{\frac{1}{q}}\\
&=\left(
\int_{\mathbb{Q}_p^{n}}\frac{1}{1-p^{-n}}\int_{|\xi|_p=1}|f(|x|_p^{-1}\xi)|^q d\xi
|x|_p^{\alpha}dx\right)^{\frac{1}{q}}\\
&=\frac{1}{(1-p^{-n})^{\frac{1}{q}}}\left(
\int_{\mathbb{Q}_p^{n}}\int_{|y|_p=|x|_p}|f(y)|^q dy
|x|_p^{\alpha-n}dx\right)^{\frac{1}{q}}\\
&=\frac{1}{(1-p^{-n})^{\frac{1}{q}}}\left(
\int_{\mathbb{Q}_p^{n}}\int_{|x|_p=|y|_p}|x|_p^{\alpha-n}dx|f(y)|^q dy
\right)^{\frac{1}{q}}\\
&=\|f\|_{L^q(|x|_p^{\alpha}dx)}.
\end{split}\end{equation*}
Therefore,
$$\frac{\|\mathcal{H}^pf\|_{L^q(|x|_p^{\alpha}dx)}}{\|f\|_{L^q(|x|_p^{\alpha}dx)}}\leq
\frac{\|\mathcal{H}^pg\|_{L^q(|x|_p^{\alpha}dx)}}{\|g\|_{L^q(|x|_p^{\alpha}dx)}},$$
which implies that the claim. In the following, without loss of generality, we may assume that
$f\in L^q(|x|_p^{\alpha}dx)$ satisfies that $f(x)=f(|x|_p^{-1})$.

By changing of variables $t=|x|_p^{-1}y$ , we have
\begin{equation*}\begin{split}
\|\mathcal{H}^pf\|_{L^q(|x|_p^{\alpha}dx)}&=\left(\int_{\mathbb{Q}_{p}^{n}}
\left|\frac{1}{|x|_p^{n}}\int_{B(0,|x|_p)}f(t)dt\right|^q|x|_p^{\alpha}dx\right)^{\frac{1}{q}}\\
&=\left(\int_{\mathbb{Q}_{p}^{n}}
\left|\int_{B(0,1)}f(|x|_p^{-1}y)dy\right|^q|x|_p^{\alpha}dx\right)^{\frac{1}{q}}\\
\end{split}\end{equation*}
Then using Minkowski's integral inequality, we get
\begin{equation*}\begin{split}
\|\mathcal{H}^pf\|_{L^q(|x|_p^{\alpha}dx)}
&\leq \int_{B(0,1)}\left(\int_{\mathbb{Q}_{p}^{n}}\left|f(|y|_p^{-1}x)\right|^q|x|_p^{\alpha}dx\right)^{\frac{1}{q}}dy\\
&\leq \left(\int_{B(0,1)}|y|_p^{-\frac{n}{q}-\frac{\alpha}{q}}dy\right)
\|f\|_{L^q(|x|_p^{\alpha}dx)}\\
&=\left(\sum_{k=0}^{\infty}\int_{|y|_p=p^{-k}}p^{\frac{k(n+\alpha)}{q}}dy\right)\|f\|_{L^q(|x|_p^{\alpha}dx)}\\
&=\frac{1-p^{-n}}{1-p^{\frac{\alpha}{q}-\frac{n}{q'}}}\|f\|_{L^q(|x|_p^{\alpha}dx)},
\end{split}\end{equation*}
where $\frac{1}{q}+\frac{1}{q'}=1$. Therefore, we get
\begin{equation}
\|\mathcal{H}^p\|_{L^q(|x|_p^{\alpha}dx)\rightarrow L^q(|x|_p^{\alpha}dx)}
\leq\frac{1-p^{-n}}{1-p^{\frac{\alpha}{q}-\frac{n}{q'}}}.\label{less}
\end{equation}

On the other hand, for $0<\epsilon<1$, we take
\begin{equation*}
f_{\epsilon}=\left\{
       \begin{array}{ccc}
       0,&\quad&|x|_p<1, \\
      |x|_p^{-\frac{n}{q}-\frac{\alpha}{q}-\epsilon},&\quad&|x|_p\geq1.
        \end{array}\right.
\end{equation*}
Then $\|f_{\epsilon}\|^q_{L^q(|x|_p^{\alpha}dx)}=\frac{1-p^{-n}}{1-p^{-\epsilon q}}$, and
\begin{equation*}
\mathcal{H}^pf_{\epsilon}(x)=\left\{
       \begin{array}{ccc}
       0,&\quad&|x|_p<1, \\
      |x|_p^{-\frac{n}{q}-\frac{\alpha}{q}-\epsilon}\int_{\frac{1}{|x|_p}\leq|t|_p\leq 1}
      |t|_p^{-\frac{n}{q}-\frac{\alpha}{q}-\epsilon}dt,&\quad&|x|_p\geq1.
        \end{array}\right.
\end{equation*}

Assume $|\epsilon|_p>1$. We have
\begin{equation}\begin{split}
\|\mathcal{H}^pf_{\epsilon}\|_{L^q(|x|_p^{\alpha}dx)}&=\left\{\int_{|x|_{p}\geq 1}
\left(|x|_p^{-\frac{n}{q}-\frac{\alpha}{q}-\epsilon}
\int_{\frac{1}{|x|_p}\leq|t|_p\leq 1}|t|_p^{-\frac{n}{q}-\frac{\alpha}{q}-\epsilon}dt\right)^q
|x|_p^{\alpha}dx\right\}^{\frac{1}{q}}\\
&\geq\left\{\int_{|x|_{p}\geq |\epsilon|_p}
\left(|x|_p^{-\frac{n}{q}-\frac{\alpha}{q}-\epsilon}
\int_{\frac{1}{|\epsilon|_p}\leq|t|_p\leq 1}|t|_p^{-\frac{n}{q}-\frac{\alpha}{q}-\epsilon}dt\right)^q
|x|_p^{\alpha}dx\right\}^{\frac{1}{q}}\\
&= \left(\int_{|x|_{p}\geq |\epsilon|_p}
|x|_p^{-n-\epsilon q}dx\right)^{\frac{1}{q}}
\int_{\frac{1}{|\epsilon|_p}\leq|t|_p\leq 1}|t|_p^{-\frac{n}{q}-\frac{\alpha}{q}-\epsilon}dt
\\\end{split}\end{equation}
\begin{equation*}\begin{split}
&=\|f_{\epsilon}\|_{L^q(|x|_p^{\alpha}dx)}|\epsilon|_p^{-\epsilon}
\int_{\frac{1}{|\epsilon|_p}\leq|t|_p\leq 1}|t|_p^{-\frac{n}{q}-\frac{\alpha}{q}-\epsilon}dt.
\end{split}\end{equation*}
Therefore,
\begin{equation}
\int_{\frac{1}{|\epsilon|_p}\leq|t|_p\leq 1}|t|_p^{-\frac{n}{q}-\frac{\alpha}{q}-\epsilon}dt
\leq\|\mathcal{H}^p\|_{L^q(|x|_p^{\alpha}dx)\rightarrow L^q(|x|_p^{\alpha}dx)}
|\epsilon|_p^{\epsilon}.
\end{equation}

Now take $\epsilon=p^{-k},~k=1,2,3,\cdots$. Then $|\epsilon|_p=p^{k}>1$. Letting $k$ approach to $\infty$,
then $\epsilon$ approaches to $0$ and $|\epsilon|_p^{\epsilon}=p^{\frac{k}{p^{k}}}$ approaches to $1$. Then by
Fatou's Lemma, we obtain
\begin{equation}
\frac{1-p^{-n}}{1-p^{\frac{\alpha}{q}-\frac{n}{q'}}}=
\int_{0<|t|_p\leq 1}|t|_p^{-\frac{n}{q}-\frac{\alpha}{q}}dt
\leq\|\mathcal{H}^p\|_{L^q(|x|_p^{\alpha}dx)\rightarrow L^q(|x|_p^{\alpha}dx)}.\label{more}
\end{equation}
Then \eqref{less} and \eqref{more} imply that $$\|\mathcal{H}^p\|_{L^q(|x|_p^{\alpha}dx)\rightarrow L^q(|x|_p^{\alpha}dx)}=\frac{1-p^{-n}}{1-p^{\frac{\alpha}{q}-\frac{n}{q'}}}.$$

By the similar method, we can also obtain that $$\|\mathcal{H}^{p,*}\|_{L^q(|x|_p^{\alpha}dx)\rightarrow L^q(|x|_p^{\alpha}dx)}=\frac{1-p^{-n}}{1-p^{\frac{\alpha}{q}-\frac{n}{q'}}}.$$
 Theorem \ref{th:thm1} is proved.
\end{proof}

\begin{proof}[Proof of Theorem \ref{th:thm3}]
By definition and the change of variables $y=xt$, we have
\begin{equation*}\begin{split}
\|T^pf\|_{L^q(|x|_p^{\alpha}dx)}&=\left(\int_{\mathbb{Q}_{p}}
\left|\int_{\mathbb{Q}_{p}^{*}}\frac{f(y)}{\max(|x|_p,|y|_p)}dy\right|^q|x|_p^{\alpha}dx\right)^{\frac{1}{q}}\\
&\leq\left(\int_{\mathbb{Q}_{p}}
\left(\int_{\mathbb{Q}_{p}^{*}}\frac{|f(y)|}{\max(|x|_p,|y|_p)}dy\right)^q|x|_p^{\alpha}dx\right)^{\frac{1}{q}}\\
&=\left(\int_{\mathbb{Q}_{p}}
\left(\int_{\mathbb{Q}_{p}^{*}}\frac{|f(xt)|}{\max(1,|t|_p)}dt\right)^q|x|_p^{\alpha}dx\right)^{\frac{1}{q}}.
\end{split}\end{equation*}

By Minkowski's integral inequality, we get
\begin{equation}\begin{split}
\|T^pf\|_{L^q(|x|_p^{\alpha}dx)}&\leq
\int_{\mathbb{Q}_{p}^{*}}\left(\int_{\mathbb{Q}_{p}}|f(xt)|^q|x|_p^{\alpha}dx\right)^{\frac{1}{q}}
\frac{1}{\max(1,|t|_p)}dt\\
&\leq \int_{\mathbb{Q}_{p}^{*}}\left(\int_{\mathbb{Q}_{p}}|f(x)|^q|x|_p^{\alpha}dx\right)^{\frac{1}{q}}
\frac{|t|_p^{-\frac{\alpha+1}{q}}}{\max(1,|t|_p)}dt\\
&=\|f\|_{L^q(|x|^{\alpha}_pdx)}\int_{\mathbb{Q}_{p}^{*}}
\frac{|t|_p^{-\frac{\alpha+1}{q}}}{\max(1,|t|_p)}dt.\label{Tf1}
\end{split}\end{equation}

Since
\begin{equation}\begin{split}
\int_{\mathbb{Q}_{p}^{*}}
\frac{|t|_p^{-\frac{\alpha+1}{q}}}{\max(1,|t|_p)}dt
&=\sum_{k=-\infty}^{0}\int_{S_k}
\frac{|t|_p^{-\frac{\alpha+1}{q}}}{\max(1,|t|_p)}dt+
\sum_{k=1}^{+\infty}\int_{S_k}
\frac{|t|_p^{-\frac{\alpha+1}{q}}}{\max(1,|t|_p)}dt\\
&=\left(1-\frac{1}{p}\right)\left(\sum_{k=-\infty}^{0}p^{k(1-\frac{\alpha+1}{q})}
+
\sum_{k=1}^{+\infty}p^{-\frac{k(\alpha+1)}{q}}\right)\\
&=\left(1-\frac{1}{p}\right)
\left(\frac{1}{1-p^{\frac{\alpha+1}{q}-1}}
+\frac{p^{-\frac{\alpha+1}{q}}}{1-p^{-\frac{\alpha+1}{q}}}\right).\label{norm}
\end{split}\end{equation}
Substituting \eqref{norm} into \eqref{Tf1} shows that \eqref{eq} holds and
\begin{equation}
\|T^p\|_{L^q(|x|^{\alpha}_pdx)\rightarrow L^q(|x|^{\alpha}_pdx)}
\leq\left(1-\frac{1}{p}\right)
\left(\frac{1}{1-p^{\frac{\alpha+1}{q}-1}}
+\frac{p^{-\frac{\alpha+1}{q}}}{1-p^{-\frac{\alpha+1}{q}}}\right).\label{less1}
\end{equation}

On the other hand, for $0<\epsilon<1$, let
\begin{equation*}
f_{\epsilon}=\left\{
       \begin{array}{ccc}
       0,&\quad&|x|_p<1, \\
      |x|_p^{\frac{-1-\alpha}{q}-\epsilon},&\quad&|x|_p\geq1.
        \end{array}\right.
\end{equation*}
Then $\|f_{\epsilon}\|^q_{L^q(|x|_p^{\alpha}dx)}=\frac{1-p^{-1}}{1-p^{-\epsilon q}}$, and
\begin{equation*}
T^pf_{\epsilon}(x)=
      \int_{|y|_p\geq1}
      \frac{|y|_p^{\frac{-1-\alpha}{q}-\epsilon}}{\max(|x|_p,|y|_p)}dy.
\end{equation*}
Set $|\epsilon|_p>1$. We have
\begin{equation}\begin{split}
\|T^pf_{\epsilon}\|_{L^q(|x|_p^{\alpha}dx)}&=\left\{\int_{\mathbb{Q}_{p}}
\left( \int_{|y|_p\geq1}
      \frac{|y|_p^{\frac{-1-\alpha}{q}-\epsilon}}{\max(|x|_p,|y|_p)}dy\right)^q
|x|_p^{\alpha}dx\right\}^{\frac{1}{q}}\\
&\geq\left\{\int_{|x|_{p}\geq |\epsilon|_p}
\left( \int_{|t|_p\geq \frac{1}{|\epsilon|_p}}
      \frac{|t|_p^{\frac{-1-\alpha}{q}-\epsilon}}{\max(1,|t|_p)}dt\right)^q
|x|_p^{-1-\epsilon q}dx\right\}^{\frac{1}{q}}\\
&=\|f_{\epsilon}\|^q_{L^q(|x|_p^{\alpha}dx)}|\epsilon|_p^{-\epsilon}
\int_{|t|_p\geq \frac{1}{|\epsilon|_p}}
      \frac{|t|_p^{\frac{-1-\alpha}{q}-\epsilon}}{\max(1,|t|_p)}dt.
\end{split}\end{equation}
Therefore,
\begin{equation}
\int_{|t|_p\geq \frac{1}{|\epsilon|_p}}
      \frac{|t|_p^{\frac{-1-\alpha}{q}-\epsilon}}{\max(1,|t|_p)}dt
\leq\|T^p\|_{L^q(|x|_p^{\alpha}dx)\rightarrow L^q(|x|_p^{\alpha}dx)}
|\epsilon|_p^{\epsilon}.
\end{equation}

Now take $\epsilon=p^{-k},~k=1,2,3,\cdots$. Then $|\epsilon|_p=p^{k}>1$. Letting $k$ approach to $\infty$ and
using
Fatou's Lemma, we obtain

\begin{equation}\begin{split}
\left(1-\frac{1}{p}\right)
\left(\frac{1}{1-p^{\frac{\alpha+1}{q}-1}}
+\frac{p^{-\frac{\alpha+1}{q}}}{1-p^{-\frac{\alpha+1}{q}}}\right)&=
\int_{\mathbb{Q}_{p}^{*}}
\frac{|t|_p^{-\frac{\alpha+1}{q}}}{\max(1,|t|_p)}dt\\
&\leq\|T^p\|_{L^q(|x|_p^{\alpha}dx)\rightarrow L^q(|x|_p^{\alpha}dx)}.\label{more1}
\end{split}\end{equation}
Then \eqref{norm1} follows from \eqref{less1} and \eqref{more1}. Theorem \ref{th:thm3} is proved.
\end{proof}

\section{Boundedness of commutators of $p$-adic Hardy and Hardy-Littlewood-P\'olya operators}
The boundedness of commutators is an active topic in harmonic analysis due to its important applications, for example,
 it can be applied to characterizing some function spaces. There are a lot of works about the boundedness
of commutators of various Hardy-type operators on Euclidean spaces (cf. \cite{Fu2}, \cite{Fu3}).
In this section, we consider the boundedness of commutators of $p$-adic Hardy
and Hardy-Littlewood-P\'olya operators.

\begin{defn}
Let $b\in L_{loc}(\mathbb{Q}_{p}^{n})$. {\it The commutators of  $p$-adic Hardy
operators} are defined by
\begin{equation}
\mathcal{H}^p_{b}f=b\mathcal{H}^p f-\mathcal{H}^p(bf),\qquad
\mathcal{H}^{p,*}_{b}f=b\mathcal{H}^{p,*} f-\mathcal{H}^{p,*}(bf).
\end{equation}
\end{defn}

\begin{defn}
Let $b\in L_{loc}(\mathbb{Q}_{p}^{n})$. {\it The commutator of  $p$-adic
Hardy-Littlewood-P\'olya operator} is defined by
\begin{equation}
T^p_{b}f=bT^p f-T^p(bf).
\end{equation}
\end{defn}

In \cite{CL}, \cite{GV} and \cite{Lu}, the CMO spaces (central BMO spaces) on $\mathbb{R}^n$ have been introduced and studied.
CMO spaces bears a simple relationship with BMO: $g\in BMO$ precisely when $g$
and all of its translates belong to BMO spaces uniformly a.e.. Many precise analogies
exist between CMO spaces and BMO space from the point of view of real Hardy spaces.
Similarly, we define the
$CMO^q$ spaces on $\mathbb{Q}_{p}^{n}$.

\begin{defn}
Let $1\leq q<\infty$. A function $f\in L^q_{loc}(\mathbb{Q}_{p}^{n})$ is said to be in $CMO^q(\mathbb{Q}_{p}^{n})$,
if
\begin{equation*}
\|f\|_{CMO^q(\mathbb{Q}_{p}^{n})}:=
\sup_{\gamma\in\mathbb{Z}}\left(\frac{1}{|B_\gamma(0)|_H}
\int_{B_{\gamma(0)}}|f(x)-f_{B_{\gamma(0)}}|^qdx\right)^{\frac{1}{q}}
<\infty,
\end{equation*}
where
\begin{equation*}
f_{B_{\gamma(0)}}=\frac{1}{|B_\gamma(0)|_H}\int_{B_{\gamma(0)}}f(x)dx.
\end{equation*}
\end{defn}

\begin{rem}
It is obvious that $L^\infty(\mathbb{Q}_{p}^{n})\subset BMO(\mathbb{Q}_{p}^{n})\subset CMO^q(\mathbb{Q}_{p}^{n})$.
\end{rem}

Since Herz space is a natural generalization of Lebesgue space with power weight, we further study the boundedness
of commutators of $p$-adic Hardy
and Hardy-Littlewood-P\'olya operators on Herz space. Let us first give the definition of Herz space.

Let $B_k=B_k(0)=\{x\in\mathbb{Q}_{p}^{n}:|x|_p\leq p^{k}\}$, $S_k=B_k\setminus B_{k-1}$
and $\chi_E$ is the characteristic function of set $E$.

\begin{defn}{\rm\cite{Zh}}
Suppose $\alpha\in\mathbb{R}$, $0<q<\infty$ and $0< r<\infty$. The homogeneous $p$-adic Herz
space $K^{\alpha, q}_{r}(\mathbb{Q}_{p}^{n})$ is defined by
\begin{equation*}
K^{\alpha, q}_{r}(\mathbb{Q}_{p}^{n})=\left\{f\in L^r_{loc}(\mathbb{Q}_{p}^{n}):
\|f\|_{K^{\alpha, q}_{r}(\mathbb{Q}_{p}^{n})}<\infty\right\},
\end{equation*}
where
\begin{equation*}
\|f\|_{K^{\alpha, q}_{r}(\mathbb{Q}_{p}^{n})}=\left(\sum_{k=-\infty}^{+\infty}
p^{k\alpha q}\|f\chi_k\|^q_{ L^r(\mathbb{Q}_{p}^{n})}\right)^{\frac{1}{q}},
\end{equation*}
with the usual modifications made when $q=\infty$ or $r=\infty$.
\end{defn}
\begin{rem}\label{th:rem1}
$K^{0, q}_{q}(\mathbb{Q}_{p}^{n})$ is the generalization of $L^q(|x|^{\alpha}_{p}dx)$, and
  $K^{\frac{\alpha}{q}, q}_{q}(\mathbb{Q}_{p}^{n})= L^q(|x|^{\alpha}_{p}dx)$,
 $K^{0, q}_{q}(\mathbb{Q}_{p}^{n})= L^q(\mathbb{Q}_{p}^{n})$ for
all $0<q\leq\infty$ and $\alpha\in\mathbb{R}$.
\end{rem}

In what follows, we will not study the best estimates of the two commutators mentioned above,  instead,
we will discuss the boundedness of them.
Motivated by \cite{Fu2}, we get the following operator boundedness results. Throughout
this paper, we use $C$ to denote different positive constants which are independent of the essential variables,
and $q'$ to denote the conjugate index of $q$, i.e. $\frac{1}{q}+\frac{1}{q'}=1$.

\begin{thm}\label{th:thm2}
Let $0<q_1\leq q_2<\infty$, $1<r<\infty$
and $b\in CMO^{\max\{r',r\}}(\mathbb{Q}_{p}^{n})$. Then\\
$(1)$ if $\alpha<\frac{n}{r'}$, then
$\mathcal{H}^p_{b}$ is bounded from $K^{\alpha, q_1}_{r}(\mathbb{Q}_{p}^{n})$
to $K^{\alpha, q_2}_{r}(\mathbb{Q}_{p}^{n})$;\\
$(2)$ if $\alpha>-\frac{n}{r}$, then $\mathcal{H}^{p,*}_{b}$ is bounded from $K^{\alpha, q_1}_{r}(\mathbb{Q}_{p}^{n})$
to $K^{\alpha, q_2}_{r}(\mathbb{Q}_{p}^{n})$.
\end{thm}

 From Remark \ref{th:rem1}, we get the following two corollaries.

\begin{cor}
Suppose that $1<q<\infty$
and $b\in CMO^{\max\{q',q\}}(\mathbb{Q}_{p}^{n})$.  Then\\
$(1)$ if $\alpha<\frac{nq}{q'}$, then
$\mathcal{H}^p_{b}$ is bounded from $L^q(|x|_p^{\alpha}dx)$
to $L^q(|x|_p^{\alpha}dx)$;\\
$(2)$ if $\alpha>-n$, then $\mathcal{H}^{p,*}_{b}$ is bounded from $L^q(|x|_p^{\alpha}dx)$
to $L^q(|x|_p^{\alpha}dx)$.
\end{cor}

\begin{cor}
Suppose that $1<q<\infty$
 and $b\in CMO^{\max\{q',q\}}(\mathbb{Q}_{p}^{n})$. Then
both $\mathcal{H}^p_{b}$ and $\mathcal{H}^{p,*}_{b}$ are bounded from $L^q(\mathbb{Q}_{p}^{n})$
to $L^q(\mathbb{Q}_{p}^{n})$.
\end{cor}

Due to Theorem \ref{th:thm2}, we can also obtain the boundedness
of commutator generated by
$p$-adic Hardy-Littlewood-P\'olya operator and CMO function.

\begin{thm}\label{th:thm4}
Suppose that $0<q_1\leq q_2<\infty$, $1<r<\infty$,
$-\frac{1}{r}<\alpha<\frac{1}{r'}$ and $b\in CMO^{\max\{r',r\}}(\mathbb{Q}_{p}^{n})$. Then
 $T^p_{b}$ is bounded from $K^{\alpha, q_1}_{r}(\mathbb{Q}_{p}^{n})$
to $K^{\alpha, q_2}_{r}(\mathbb{Q}_{p}^{n})$.
\end{thm}

\begin{cor}
Suppose that $1<q<\infty$,
 $-1<\alpha<q-1$ and $b\in CMO^{\max\{q',q\}}(\mathbb{Q}_{p}^{n})$. Then
 $T^p_{b}$ and $T^{p,*}_{b}$ is bounded from $L^q(|x|_p^{\alpha}dx)$
to $L^q(|x|_p^{\alpha}dx)$.
\end{cor}

\begin{cor}
Suppose that $1<q<\infty$
 and $b\in CMO^{\max\{q',q\}}(\mathbb{Q}_{p}^{n})$. Then
 $T^p_{b}$ is bounded from $L^q(\mathbb{Q}_{p}^{n})$
to $L^q(\mathbb{Q}_{p}^{n})$.
\end{cor}

To prove Theorem \ref{th:thm2}, we need the following lemmas.

\begin{lem}\label{th:lem1}
Suppose that $b$ is a $CMO$ function and $1\leq q<r<\infty$. Then
$CMO^r(\mathbb{Q}_{p}^{n})\subset CMO^q(\mathbb{Q}_{p}^{n})$ and
$\|b\|_{CMO^q(\mathbb{Q}_{p}^{n})}\leq\|b\|_{CMO^r(\mathbb{Q}_{p}^{n})}$.
\end{lem}

\begin{proof}
For any $b\in CMO^r(\mathbb{Q}_{p}^{n})$, by H\"{o}lder's inequality, we have
\begin{equation*}\begin{split}
&\left(\frac{1}{|B_\gamma(0)|_H}
\int_{B_{\gamma(0)}}|b(x)-b_{B_{\gamma(0)}}|^qdx\right)^{\frac{1}{q}}\\
&\leq\left\{\frac{1}{|B_\gamma(0)|_H}
\left(\int_{B_{\gamma(0)}}|b(x)-b_{B_{\gamma(0)}}|^{q\cdot\frac{r}{q}}dx\right)^{\frac{q}{r}}
\left(\int_{B_{\gamma(0)}}1dx\right)^{1-\frac{q}{r}}\right\}^{\frac{1}{q}}\\
&=\left\{\frac{1}{|B_\gamma(0)|_H}
\left(\int_{B_{\gamma(0)}}|b(x)-b_{B_{\gamma(0)}}|^{r}dx\right)^{\frac{q}{r}}
\left|B_{\gamma(0)}\right|_H^{1-\frac{q}{r}}\right\}^{\frac{1}{q}}\\
&=\left(\frac{1}{|B_\gamma(0)|_H}\int_{B_{\gamma(0)}}|b(x)-b_{B_{\gamma(0)}}|^{r}dx\right)^{\frac{1}{r}}\\
&\leq \|b\|_{CMO^r(\mathbb{Q}_{p}^{n})}
\end{split}\end{equation*}
Therefore, $b\in CMO^q(\mathbb{Q}_{p}^{n})$ and $\|b\|_{CMO^q(\mathbb{Q}_{p}^{n})}
\leq\|b\|_{CMO^r(\mathbb{Q}_{p}^{n})}$. This completes the proof.
\end{proof}

\begin{lem}\label{th:lem2}
Suppose that b is a $CMO$ function, $j,k\in\mathbb{Z}$. Then
\begin{equation}
|b(t)-b_{B_k}|\leq |b(t)-b_{B_k}|+p^n|j-k|\|b\|_{CMO^1(\mathbb{Q}_{p}^{n})}.\label{lem2}
\end{equation}
\end{lem}

\begin{proof}
For $i\in\mathbb{Z}$, recall that $b_{B_i}=\frac{1}{|B_i|_H}\int_{B_i}b(x)dx$, we have
\begin{equation}\begin{split}
|b_{B_i}-b_{B_{i+1}}|&\leq\frac{1}{|B_i|_H}\int_{B_i} |b(t)-b_{B_{i+1}}|dt
\leq\frac{p^n}{|B_{i+1}|_H}\int_{B_{i+1}} |b(t)-b_{B_{i+1}}|dt\\
&\leq p^n\|b\|_{CMO^1(\mathbb{Q}_{p}^{n})}.\label{bi}
\end{split}\end{equation}

For $j,k\in\mathbb{Z}$, without loss of generality, we can assume that $j\leq k$. By \eqref{bi}
we get
\begin{equation*}\begin{split}
|b(t)-b_{B_k}|\leq|b(t)-b_{B_j}|+\sum_{i=k}^{j-1}|b_{B_i}-b_{B_{i+1}}|
\leq |b(t)-b_{B_j}|+p^n|j-k|\|b\|_{CMO^1(\mathbb{Q}_{p}^{n})}.
\end{split}\end{equation*}
The lemma is proved.
\end{proof}

\begin{proof}[Proof of Theorem \ref{th:thm2}]

Denote $f(x)\chi_i(x)=f_i(x)$. By definition
\begin{equation*}\begin{split}
\|(\mathcal{H}^p_{b}f)\chi_k\|_{L^{r}(\mathbb{Q}_{p}^{n})}^{r}
&=\int_{S_k}|x|_p^{-rn}
\left|\int_{B(0,|x|_p)}f(t)(b(x)-b(t))dt\right|^{r}dx\\
&\leq \int_{S_k}p^{-krn}
\left(\int_{B(0,p^k)}\left|f(t)(b(x)-b(t)) \right|dt\right)^{r}dx\\
&=p^{-krn}\int_{S_k}\left(
\sum_{j=-\infty}^{k}\int_{S_j}\left|f(t)(b(x)-b(t)) \right|dt\right)^{r}dx\\
&\leq Cp^{-krn}\int_{S_k}\left(
\sum_{j=-\infty}^{k}\int_{S_j}\left|f(t)(b(x)-b_{B_k}) \right|dt\right)^{r}dx\\
&\quad~ Cp^{-krn}\int_{S_k}\left(
\sum_{j=-\infty}^{k}\int_{S_j}\left|f(t)(b(t)-b_{B_k})\right|dt\right)^{r}dx\\
&:=I+II.
\end{split}\end{equation*}

Now let's estimate $I$ and $J$, respectively. For $I$, by H\"{o}lder's inequality $(\frac{1}{r}+\frac{1}{r'}=1)$, we have
\begin{equation}\begin{split}
I&=Cp^{-krn}\left(\int_{S_k}|b(x)-b_{B_k}|^{r}dx\right)\left(
\sum_{j=-\infty}^{k}\int_{S_j}\left|f(t)\right|dt\right)^{r}\\
&\leq Cp^{\frac{-krn}{r'}}\left(\frac{1}{|B_k|_H}\int_{B_k}|b(x)-b_{B_k}|^{r}dx\right)\times\\
&\quad\left\{\sum_{j=-\infty}^{k}\left(\int_{S_j}\left|f(t)\right|^{r}dt\right)^{\frac{1}{r}}
\left(\int_{S_j}dt\right)^{\frac{1}{r'}}\right\}^{r}\\\label{I}
\end{split}\end{equation}\begin{equation*}\begin{split}
&\leq C\|b\|^{r}_{CMO^{r}(\mathbb{Q}_{p}^{n})}\left\{\sum_{j=-\infty}^{k}{p^{\frac{{(j-k)n}}{r'}}}\|f_j\|_{L^{r}(\mathbb{Q}_{p}^{n})}
\right\}^{r}.
\end{split}\end{equation*}

For $II$, by Lemma \ref{th:lem2}, we get
\begin{equation*}\begin{split}
II&=Cp^{-krn}\int_{S_k}\left(
\sum_{j=-\infty}^{k}\int_{S_j}\left|f(t)(b(t)-b_{B_k})\right|dt\right)^{r}dx\\
&=Cp^{-krn}p^{kn}(1-p^{-n})\left(
\sum_{j=-\infty}^{k}\int_{S_j}\left|f(t)(b(t)-b_{B_k})\right|dt\right)^{r}\\
&\leq Cp^{\frac{-krn}{r'}}\left(
\sum_{j=-\infty}^{k}\int_{S_j}\left|f(t)(b(t)-b_{B_j})\right|dt\right)^{r}\\
&\quad+Cp^{\frac{-krn}{r'}}\|b\|^{r}_{CMO^{1}(\mathbb{Q}_{p}^{n})}\left(
\sum_{j=-\infty}^{k}(k-j)\int_{S_j}\left|f(t)\right|dt\right)^{r}\\
&:=II_1+II_2.
\end{split}\end{equation*}

For $II_1$ and $II_2$, by H\"{o}lder's inequality, we obtain
\begin{equation}\begin{split}
II_1&\leq Cp^{\frac{-krn}{r'}}\left\{\left(
\sum_{j=-\infty}^{k}\int_{S_j}\left|f(t)\right|^{r}dt\right)^{\frac{1}{r}}
\left(\int_{S_j}\left|b(t)-b_{B_j}\right|^{r'}dt\right)^{\frac{1}{r'}}\right\}^{r}\\
&\leq Cp^{\frac{-krn}{r'}}\left\{
\sum_{j=-\infty}^{k}\left\|f_j\right\|_{L^{r}(\mathbb{Q}_{p}^{n})}p^{\frac{jn}{r'}}
\left(\frac{1}{|B_j|_H}\int_{B_j}\left|b(t)-b_{B_j}\right|^{r'}dt\right)^{\frac{1}{r'}}\right\}^{r}\\
&\leq C\|b\|^{r}_{CMO^{r'}(\mathbb{Q}_{p}^{n})}\left\{
\sum_{j=-\infty}^{k}p^{\frac{(j-k)n}{r'}}\left\|f_j\right\|_{L^{r}(\mathbb{Q}_{p}^{n})}\right\}^{r}.\label{II1}
\end{split}\end{equation}
And
\begin{equation}\begin{split}
II_2&\leq Cp^{\frac{-krn}{r'}}\|b\|^{r}_{CMO^{1}(\mathbb{Q}_{p}^{n})}\left\{
\sum_{j=-\infty}^{k}(k-j)\left(\int_{S_j}\left|f(t)\right|^{r}dt\right)^{\frac{1}{r}}
\left(\int_{S_j}dt\right\}^{\frac{1}{r'}}\right\}^{r}\\
&\leq C\|b\|^{r}_{CMO^{1}(\mathbb{Q}_{p}^{n})}\left\{
\sum_{j=-\infty}^{k}(k-j)p^{\frac{(j-k)n}{r'}}\left\|f_j\right\|_{L^{r}(\mathbb{Q}_{p}^{n})}\right\}^{r}.
\label{II2}
\end{split}\end{equation}

Then \eqref{I}--\eqref{II2} together with Lemma \ref{th:lem1} imply that
\begin{equation*}\begin{split}
&\|\mathcal{H}_bf\|_{K^{\alpha,q_2}_{r}(\mathbb{Q}_{p}^{n})}
=\left(\sum_{k=-\infty}^{+\infty}
p^{k\alpha q_2}\|(U_{\beta,b}f)\chi_k\|_{L^{r}(\mathbb{Q}_{p}^{n})}^{q_2}\right)^{\frac{1}{q_2}}\\
\end{split}\end{equation*}
\begin{equation}\begin{split}
&\leq \left(\sum_{k=-\infty}^{+\infty}
p^{k\alpha q_1}\|(U_{\beta,b}f)\chi_k\|_{L^{r}(\mathbb{Q}_{p}^{n})}^{q_1}\right)^{\frac{1}{q_1}}\\
&\leq C\left(\sum_{k=-\infty}^{+\infty}
p^{k\alpha q_1}\|b\|^{q_1}_{CMO^{r}(\mathbb{Q}_{p}^{n})}\left(\sum_{j=-\infty}^{k}
{p^{\frac{{(j-k)n}}{r'}}}\|f_j\|_{L^{r}(\mathbb{Q}_{p}^{n})}
\right)^{q_1}\right)^{\frac{1}{q_1}}\\
&\quad+ C\left(\sum_{k=-\infty}^{+\infty}
p^{k\alpha q_1}\|b\|^{q_1}_{CMO^{r'}(\mathbb{Q}_{p}^{n})}\left(\sum_{j=-\infty}^{k}
{p^{\frac{{(j-k)n}}{r'}}}\|f_j\|_{L^{r}(\mathbb{Q}_{p}^{n})}
\right)^{q_1}\right)^{\frac{1}{q_1}}\\
&\quad+ C\left(\sum_{k=-\infty}^{+\infty}
p^{k\alpha q_1}\|b\|^{q_1}_{CMO^{1}(\mathbb{Q}_{p}^{n})}\left(\sum_{j=-\infty}^{k}
{(k-j)p^{\frac{{(j-k)n}}{r'}}}\|f_j\|_{L^{r}(\mathbb{Q}_{p}^{n})}
\right)^{q_1}\right)^{\frac{1}{q_1}}\label{J}
\end{split}\end{equation}
\begin{equation*}\begin{split}
&\leq C\|b\|_{CMO^{\max\{r',r\}}(\mathbb{Q}_{p}^{n})}\left(\sum_{k=-\infty}^{+\infty}
p^{k\alpha q_1}\left(\sum_{j=-\infty}^{k}
{(k-j)p^{\frac{{(j-k)n}}{r'}}}\|f_j\|_{L^{r}(\mathbb{Q}_{p}^{n})}
\right)^{q_1}\right)^{\frac{1}{q_1}}\\
&:=J.
\end{split}\end{equation*}

For the case $0<q_1\leq 1$,  since $\alpha<\frac{{n}}{r'}$, we have
\begin{equation}\begin{split}
J^{q_1}
&=C\|b\|^{q_1}_{CMO^{\max\{r',r\}}(\mathbb{Q}_{p}^{n})}\sum_{k=-\infty}^{+\infty}
p^{k\alpha q_1}\left(\sum_{j=-\infty}^{k}
{(k-j)p^{\frac{{(j-k)n}}{r'}}}\|f_j\|_{L^{r}(\mathbb{Q}_{p}^{n})}
\right)^{q_1}\\
&=C\|b\|^{q_1}_{CMO^{\max\{r',r\}}(\mathbb{Q}_{p}^{n})}\sum_{k=-\infty}^{+\infty}
\left(\sum_{j=-\infty}^{k}p^{j\alpha}\|f_j\|_{L^{r}(\mathbb{Q}_{p}^{n})}
{(k-j)p^{(j-k)(\frac{{n}}{r'}-\alpha)}}
\right)^{q_1}\\
&\leq C\|b\|^{q_1}_{CMO^{\max\{r',r\}}(\mathbb{Q}_{p}^{n})}\sum_{k=-\infty}^{+\infty}
\sum_{j=-\infty}^{k}p^{j\alpha q_1}\|f_j\|^{q_1}_{L^{r}(\mathbb{Q}_{p}^{n})}
{(k-j)^{q_1}p^{(j-k)(\frac{{n}}{r'}-\alpha)q_1}}
\\
&= C\|b\|^{q_1}_{CMO^{\max\{r',r\}}(\mathbb{Q}_{p}^{n})}\sum_{j=-\infty}^{+\infty}
p^{j\alpha q_1}\|f_j\|^{q_1}_{L^{r}(\mathbb{Q}_{p}^{n})}\sum_{k=j}^{+\infty}
{(k-j)^{q_1}p^{(j-k)(\frac{{n}}{r'}-\alpha)q_1}}
\\
&= C\|b\|^{q_1}_{CMO^{\max\{r',r\}}(\mathbb{Q}_{p}^{n})}
\|f\|^{q_1}_{K^{\alpha,r}_{r}(\mathbb{Q}_{p}^{n})}.\label{J1}
\end{split}\end{equation}

For the case $q_1>1$, by H\"{o}lder's inequality, we have
\begin{equation}\begin{split}
J^{q_1}
&\leq C\|b\|^{q_1}_{CMO^{\max\{r',r\}}(\mathbb{Q}_{p}^{n})}\sum_{k=-\infty}^{+\infty}
\left(\sum_{j=-\infty}^{k}p^{j\alpha q_1}\|f_j\|^{q_1}_{L^{r}(\mathbb{Q}_{p}^{n})}
p^{\frac{(j-k)}{2}(\frac{{n}}{r'}-\alpha)q_1}\right)\times\\
&\quad\left(\sum_{j=-\infty}^{k}{(k-j)^{q'_1}p^{\frac{(j-k)}{2}(\frac{{n}}{r'}-\alpha)q'_1}}\right)^{\frac{q_1}{q'_1}}
\\\end{split}\end{equation}\begin{equation*}\begin{split}
&= C\|b\|^{q_1}_{CMO^{\max\{r',r\}}(\mathbb{Q}_{p}^{n})}\sum_{j=-\infty}^{+\infty}
p^{j\alpha q_1}\|f_j\|^{q_1}_{L^{r}(\mathbb{Q}_{p}^{n})}
\sum_{k=j}^{+\infty}p^{\frac{(j-k)}{2}(\frac{{n}}{r'}-\alpha)q_1}\\
&= C\|b\|^{q_1}_{CMO^{\max\{r',r\}}(\mathbb{Q}_{p}^{n})}
\|f\|^{q_1}_{K^{\alpha,r}_{r}(\mathbb{Q}_{p}^{n})}.\label{J2}
\end{split}\end{equation*}
Then (1) follows from \eqref{J}, \eqref{J1} and \eqref{J2}.
 By the similar method, we can get
(2). Theorem \ref{th:thm2} is proved.
\end{proof}

\begin{proof}[Proof of Theorem \ref{th:thm4}]
By definition, we have
\begin{equation}\begin{split}
\left|T^p_{b}f\right|&=
\left|\int_{\mathbb{Q}_{p}^{*}}\frac{f(y)}{\max(|x|_p,|y|_p)}\left(b(x)-b(y)\right)dy\right|\\
&\leq\left|\int_{ B(0,|x|_p)}\frac{f(y)}{|x|_p}(b(x)-b(y))dy\right|+
\left|\int_{\mathbb{Q}_{p}^{n}\setminus B(0,|x|_p)}\frac{f(y)}{|y|_p}(b(x)-b(y))dy\right|\\
&=\left|\mathcal{H}^p_{b}f\right|+\left|\mathcal{H}^{p,*}_{b}f\right|.
\end{split}\end{equation}

By Minkowski's inequality, we get
\begin{equation}
\|T^p_{b}f\|_{K^{\alpha, q_2}_{r}(\mathbb{Q}_{p}^{n})}\leq
\|\mathcal{H}^p_{b}f\|_{K^{\alpha, q_2}_{r}(\mathbb{Q}_{p}^{n})}+
\|\mathcal{H}^p_{b}f\|_{K^{\alpha, q_2}_{r}(\mathbb{Q}_{p}^{n})}.
\end{equation}
Then Theorem \ref{th:thm4} follows from Theorem \ref{th:thm2}.
\end{proof}

\vskip0.1in\parskip=0mm \baselineskip 15pt
\renewcommand{\baselinestretch}{1.15}

\end{document}